\documentclass[10pt,letterpaper]{amsart}
\usepackage[utf8]{inputenc}
\usepackage{setspace}
\usepackage{mathrsfs}
\usepackage{amssymb}
\usepackage{latexsym}
\usepackage{amsfonts}
\usepackage{amsmath}
\usepackage{eucal}
\usepackage{bm}
\usepackage{bbm}
\usepackage{graphicx}
\usepackage[english]{varioref}
\usepackage[nice]{nicefrac}
\usepackage[all]{xy}
\usepackage{amsthm}
\usepackage{amssymb,amsthm,upref,amscd}

\def\op{\operatorname}

\def\mmod{\kern-1pt\operatorname{-mod}}

\newtheorem{theorem}{Theorem}[section]
\newtheorem{lemma}[theorem]{Lemma}

\newtheorem{example}[theorem]{Example}

\theoremstyle{proposition}

\newtheorem{Prop}[theorem]{\bf Proposition}

\numberwithin{equation}{section}

\begin{document}

\title[Permutation module]{The decomposition of permutation module for infinite Chevalley groups, II }

\author{Junbin Dong}
\address{Institute of Mathematical Sciences, ShanghaiTech University, 393 Middle Huaxia Road, Pudong, Shanghai 201210, China.}
\email{dongjunbin@shanghaitech.edu.cn}
%\thanks{}

%    General info

\date{July 8, 2024}


\begin{abstract}
Let $\bf G$ be  a connected reductive algebraic group over  an algebraically closed field  $\Bbbk$ and  ${\bf B}$ be a Borel subgroup of ${\bf G}$.
In this paper we completely determine the composition factors of the permutation module $\mathbb{F}[{\bf G}/{\bf B}]$ for any field $\mathbb{F}$.
\end{abstract}

\maketitle

\section{Introduction}

Let $\bf G$ be  a connected reductive algebraic group over  an algebraically closed field  $\Bbbk$ and  ${\bf B}$ be an Borel subgroup of ${\bf G}$.
We will identify ${\bf G}$ with  ${\bf G}(\Bbbk)$ and  ${\bf B}$ with  ${\bf B}(\Bbbk)$.  Let  $\mathbb{F}$ be another field and  all the representations are over  $\mathbb{F}$. Now we just regard ${\bf G}/{\bf B}$ as a quotient set and consider the vector space $\mathbb{F} [{\bf G}/{\bf B}]$, which has a basis of the left cosets of ${\bf B}$ in ${\bf G}$. With left multiplication of the group ${\bf G}$,  $\mathbb{F} [{\bf G}/{\bf B}]$ is an $\mathbb{F}{\bf G}$-module, which is isomorphic to $\mathbb{F} {\bf G} \otimes_{\mathbb{F}  \bf B} \text{tr}$, where $\text{tr}$ denotes the one-dimensional  trivial ${\bf B}$-module. The permutation module  $\mathbb{F} [{\bf G}/{\bf B}]$ was studied in  \cite{CD1} and  \cite{CD2}  when $\Bbbk= \bar{\mathbb{F}}_q$, where $\bar{\mathbb{F}}_q$ is  the algebraically closure of  finite field $\mathbb{F}_q$ of  $q$ elements. In their determination of  the composition factors of  $\mathbb{F}[{\bf G}/{\bf B}]$,  the proofs  make  essential use of  the fact that $\bar{\mathbb{F}}_q$  is a union of finite fields.

The Steinberg module $\text{St}$ is the socle of  $\mathbb{F}[{\bf G}/{\bf B}]$ and  the irreducibility of  $\text{St}$ has been proved by Xi  (see \cite{Xi}) in the case  $\Bbbk= \bar{\mathbb{F}}_q$, and $\text{char} \  \mathbb{F}  =0 \ \text{or}\  \text{char} \  \bar{\mathbb{F}}_q$.  Later, Yang removed this restriction on $\text{char} \ \mathbb{F}$ and  proved the irreducibility of Steinberg module for any field $\mathbb{F}$ in \cite{Yang} (also in the case  $\Bbbk= \bar{\mathbb{F}}_q$).  Recently, A.Putman and A.Snowden showed that when $\Bbbk$ is an  infinite field (not necessary to be algebraically closed), then the Steinberg representation of ${\bf G}$ is always irreducible for any field $ \mathbb{F}$ (see \cite{PS}).
Their work inspires the idea of  the  determination of  the composition factors of  $\mathbb{F}[{\bf G}/{\bf B}]$ for general case  in this paper.  We will construct a filtration of submodules for  $\mathbb{F}[{\bf G}/{\bf B}]$ whose  subquotients are denoted by $E_J$  (indexed by the subsets of the set $I$ of simple reflections).   The main theorem is as follows:

\begin{theorem} \label{mainthm} \normalfont
Let $\mathbb{F}$ be any field.  All $\mathbb{F} {\bf G}$-modules $E_J$  are irreducible and pairwise non-isomorphic. Moreover, the $\mathbb{F} {\bf G}$-module $\mathbb{F}[{\bf G}/{\bf B}]$ has exactly $2^{|I|}$ composition factors, each occurring with  multiplicity one.
\end{theorem}

It is well known that the flag variety ${\bf G}/{\bf B}$ plays a very important role in the representation theory.  So the decomposition of  $\mathbb{F}[{\bf G}/{\bf B}]$ may have many applications in other areas such as algebraic geometry and number theory.

This paper is organized as follows: Section 2 contains some notations and preliminary results. In particular, we study the properties of  the subquotient modules $E_J$ of $\mathbb{F} [{\bf G}/{\bf B}]$. In Section 3, we list some properties of the unipotent radical ${\bf U}$ of ${\bf B}$ and  study the self-enclosed subgroup of ${\bf U}$,  which is useful in the later discussion. Section 4  gives the non-vanishing property of the augmentation.  In the last section, we will prove that  all the $\mathbb{F} {\bf G}$-modules $E_J$  are irreducible for any fields $\Bbbk$ and  $\mathbb{F}$.

\section{Preliminaries}

As in the introduction,  $\bf G$ is  a connected reductive algebraic group over  an algebraically closed field  $\Bbbk$ and  ${\bf B}$ is a Borel subgroup. Let  ${\bf T}$ be a  maximal torus contained in ${\bf B}$, and ${\bf U}=R_u({\bf B})$ be the  unipotent radical of ${\bf B}$. We identify ${\bf G}$ with ${\bf G}(\Bbbk)$ and do likewise for various subgroups of ${\bf G}$ such as ${\bf B}, {\bf T}, {\bf U}$ $\cdots$. We denote by $\Phi=\Phi({\bf G};{\bf T})$ the corresponding root system, and by $\Phi^+$ (resp. $\Phi^-$) the set of positive (resp. negative) roots determined by ${\bf B}$. Let $W=N_{\bf G}({\bf T})/{\bf T}$ be the corresponding Weyl group. We denote by $\Delta=\{\alpha_i\mid i\in I\}$ the set of simple roots and by $S=\{s_i:=s_{\alpha_i}\mid i\in I\}$ the corresponding simple reflections in $W$. For each $\alpha\in\Phi$, let ${\bf U}_\alpha$ be the root subgroup corresponding to $\alpha$ and we fix an isomorphism $\varepsilon_\alpha:  \Bbbk \rightarrow{\bf U}_\alpha$ such that $t\varepsilon_\alpha(c)t^{-1}=\varepsilon_\alpha(\alpha(t)c)$ for any $t\in{\bf T}$ and $c\in \Bbbk $. For any $w\in W$, let ${\bf U}_w$ (resp. ${\bf U}_w'$) be the subgroup of ${\bf U}$ generated by all ${\bf U}_\alpha$  with $w(\alpha)\in\Phi^-$ (resp. $w(\alpha)\in\Phi^+$).
 For any $J\subset I$, let $W_J$ be the corresponding standard parabolic subgroup of $W$ and $w_J$ be the longest element in $W_J$.
For a subgroup $H$ of $ {\bf G}$ and  $g\in {\bf G}$, let $H^{g}= g^{-1}H g$.

The permutation module  $\mathbb{F}[{\bf G}/{\bf B}]$ is  isomorphic to  the induced module  $\mathbb{M}(\op{tr})= \mathbb{F} {\bf G} \otimes_{\mathbb{F}  \bf B} \text{tr}$. Now let ${\bf 1}_{tr}$ be a nonzero element of $\op{tr}$.   For convenience, we abbreviate $x\otimes{\bf 1}_{\op{tr}}\in\mathbb{M}(\op{tr})$ to $x{\bf 1}_{\op{tr}}$.  Each element $\varphi\in\op{End}_{\mathbb{F}{\bf G}}(\mathbb{M}(\op{tr}))$  is determined by  $\varphi({\bf 1}_{tr})$. Note that $\varphi({\bf 1}_{tr})$ is a ${\bf B}$-stable vector. Thus we have $\varphi({\bf 1}_{tr})= \lambda {\bf 1}_{tr}$ for some $\lambda \in \mathbb{F}$, which implies that  $\op{End}_{\mathbb{F}{\bf G}}(\mathbb{M}(\op{tr}))\cong \mathbb{F} $.  In particular, the $\mathbb{F} {\bf G}$-module $\mathbb{M}(\op{tr})$ is indecomposable.

For any $w\in W$, let  $\dot{w}$ be a representative of $w$. For any $t\in {\bf T}$ and $n\in N_{\bf G}({\bf T})$, we have $nt  {\bf 1}_{\op{tr}}= n {\bf 1}_{\op{tr}}$. Thus $w {\bf 1}_{\op{tr}}=\dot{w}   {\bf 1}_{\op{tr}}$ is well-defined. For any $J\subset I$, we set
$$\eta_J=\sum_{w\in W_J}(-1)^{\ell(w)}w {\bf 1}_{\op{tr}},$$
where $\ell(w)$ is the length of $w$. Let $\mathbb{M}(\op{tr})_J=\mathbb{F}{\bf G}\eta_J$. It was proved in \cite[Proposition 2.3]{Xi} that
$$\mathbb{M}(\op{tr})_J=\mathbb{F}{\bf U}W\eta_J.$$
For $w\in W$,  set  $\mathscr{R}(w)=\{i\in I\mid ws_i<w\}$.  For any subset $J\subset I$, we let
$$ X_J =\{x\in W\mid x~\text{has~minimal~length~in}~xW_J\}.$$

\begin{Prop}\label{MJ=KUW}
 For any $J\subset I$, the $\mathbb{F} {\bf G}$-module $\mathbb{M}(\op{tr})_J$ has the form $$\mathbb{M}(\op{tr})_J=\sum_{w\in X_J}\mathbb{F}{\bf U}w\eta_J=\sum_{w\in X_J}\mathbb{F}{\bf U}_{w_Jw^{-1}}w\eta_J,$$
 and the set $\{uw\eta_J \mid w\in X_J, u\in {\bf U}_{w_Jw^{-1}} \}$ forms a basis of $\mathbb{M}(\op{tr})_J$.
\end{Prop}

\begin{proof}  Firstly, it is easy to see that $\mathbb{M}(\op{tr})_J=\mathbb{F}{\bf U}W\eta_J= \mathbb{F}{\bf U}X_J\eta_J$ since $y\eta_J= (-1)^{\ell(y)} \eta_J$ for any $y\in W_J$.  Let $w\in X_J$.  For any $\gamma \in  \Phi^+$ such that $w_Jw^{-1}(\gamma) \in  \Phi^+ $,  we have  $x^{-1}w^{-1}(\gamma) \in  \Phi^+$ for any $x\in W_J$. For  $u\in {\bf U}_{\gamma}$ and $x\in W_J$, we get
$$ uwx {\bf 1}_{\op{tr}}= wx (x^{-1}w^{-1} u w x ) {\bf 1}_{\op{tr}}= wx  {\bf 1}_{\op{tr}}$$
since $x^{-1}w^{-1} u w x\in {\bf U}$.  In particular, we get ${\bf U}w\eta_J= {\bf U}_{w_Jw^{-1}}w\eta_J$. Then we obtain the first part.

 In the following we show that $\{uw\eta_J \mid w\in X_J, u\in {\bf U}_{w_Jw^{-1}} \}$ forms a basis of $\mathbb{M}(\op{tr})_J$. It is enough to prove that this set is linearly independent. Suppose this set is  linearly dependent, then there exist
  $f_{u,w}\in \mathbb{F}$ (not all zero)   such that
  \begin{equation}\label{eq1}
    \sum_{w\in X_J} \sum_{u\in {\bf U}_{w_Jw^{-1}} } f_{u, w} uw\eta_J=0.
  \end{equation}
 Let $z\in X_J$  whose  length is maximal such that   $f_{u_0,z}\ne 0$ for some $u_0\in {\bf U}_{{w_J}z^{-1}}$.
 Substitute $\eta_J=\displaystyle \sum_{x\in W_J}(-1)^{\ell(x)}x {\bf 1}_{\op{tr}} $ in the equation (\ref{eq1}).   According to  the  Bruhat decomposition, the set $\{uw{\bf 1}_{\op{tr}} \mid w\in W, u\in {\bf U}_{w^{-1}} \}$ is linearly independent
in $\mathbb{M}(\op{tr})$.  Then we have
$$\sum_{u\in {\bf U}_{{w_J}z^{-1}} } f_{u, z} uzw_J{\bf 1}_{\op{tr}}=0.$$
So we get  $f_{u,z}=0$ for all $u\in {\bf U}_{{w_J}z^{-1}}$, which is a contradiction. The proposition is proved.
\end{proof}

For any $i\in I$, set $ {\bf U}_{\alpha_i}^*= {\bf U}_{\alpha_i}\backslash\{id\}$, where $id$ is the neutral element of $ {\bf U}$.  For the convenience of later discussion, we give some details about the expression of the element $\dot{s_i} u_i w \eta_J $, where
 $u_i\in {\bf U}_{\alpha_i}^*$ and $w\in X_J$. For each $u_i\in {\bf U}_{\alpha_i}^*$, we have
$$\dot{s_i}u_i\dot{s_i}=f_i(u_i)\dot{s_i}h_i(u_i)g_i(u_i),$$
where $f_i(u_i),g_i(u_i) \in {\bf U}_{\alpha_i}^*$, and $h_i(u_i)\in {\bf T}$ are uniquely determined. Moreover if we regard $f_i$ as a morphism on ${\bf U}_{\alpha_i}^*$, then  $f_i$ is a bijection. The following lemma is very useful in the later discussion. Its proof can be found in the proof of  \cite[Proposition 2.3]{Xi} and we omit it.

\begin{lemma} \label{suhlemma}
Let $u_i\in {\bf U}_{\alpha_i}^* $, with the notation above, then we have

\noindent $(a)$ If $ww_J\leq s_iww_J$, then $\dot{s_i} u_i w\eta_J =s_iw \eta_J $.

\noindent $(b)$  If $s_i w \leq  w$, then $\dot{s_i} u_i w \eta_J  = f_i(u_i)w\eta_J$.

\noindent  $(c)$  If $w \leq  s_iw$ but $s_iww_J\leq ww_J$, then $\dot{s_i} u_i w \eta_J =(f_i(u_i)-1)w\eta_J $.
\end{lemma}

Following \cite[2.6]{Xi}, we define
$$E_J=\mathbb{M}(\op{tr})_J/\mathbb{M}(\op{tr})_J',$$
where $\mathbb{M}(\op{tr})_J'$ is the sum of all $\mathbb{M}(\op{tr})_K$ with $J\subsetneq K$.  We denote by $C_J$ the image of $\eta_J$ in $E_J$.
For each $w\in W$, let
$$h_w=\sum_{y\leq w}(-1)^{\ell(w)-\ell(y)}P_{y,w}(1)y\in \mathbb{F} W,$$
where $P_{y,w}$ are Kazhdan-Lusztig polynomials (see \cite[Theorem1.1]{KL}). The set $\{h_w\mid w\in W\}$ is  a basis of $ \mathbb{F} W$.
We set
$$Y_J =\{w\in X_J \mid \mathscr{R}(ww_J)=J\}.$$

\begin{lemma}\label{wCJ}
Let $J\subset I$. Then each one of the following sets is a basis of $\mathbb{F} Wh_{w_J}$:

\noindent $(a)$ $\{wh_{w_J}\mid w\in X_J\}$;

\noindent $(b)$ $\{h_{ww_J}\mid w\in X_J\}$;

\noindent $(c)$ $\{yh_{w_J}\mid y\in Y_J\}\cup\{h_{xw_J}\mid x\in X_J\backslash Y_J\}$.
\end{lemma}

\begin{proof}
(a) By \cite[Lemma 2.6 (vi)]{KL}, we see that $$h_{w_J}=(-1)^{\ell(w_J)} \sum_{y\in W_J} (-1)^{\ell(y)}y \in \mathbb{F} W.$$
It is clear that $wh_{w_J}=(-1)^{\ell(w)}h_{w_J} $ for any $w\in W_J$. So we have $\mathbb{F} Wh_{w_J}=\mathbb{F} X_Jh_{w_J}$.
Now suppose that there exist $a_w\in \mathbb{F}$ (not all zero) such that
$$\sum_{w\in X_J} a_w wh_{w_J}=0.$$
Let $z\in X_J$ whose length is maximal such that $a_z\ne 0$. Substitute $h_{w_J}$ and we get
$a_zww_J=0$ in $\mathbb{F} W$. So $a_z=0$, which is a contraction. Therefore  $\{wh_{w_J}\mid w\in X_J\}$  is a basis of $\mathbb{F} Wh_{w_J}$.

(b) By \cite[Lemma 2.8 (c)]{G}, for $x\in X_J$, we have
  \begin{equation}\label{eq2}
  h_{xw_J}=xh_{w_J}+ \sum_{ w\in X_J, w<x}b_wwh_{w_J}, \quad  b_w\in \mathbb{F}.
  \end{equation}
Using induction on $\ell(x)$ we see that
  \begin{equation}\label{eq3}
  xh_{w_J}=h_{xw_J} + \sum_{ w\in X_J, w<x}b'_wh_{ww_J}, \quad  b'_w\in \mathbb{F}.
    \end{equation}
Thus (b) is proved by (a).

(c) We claim  that for any $w\in X_J$,  $wh_{w_J}$ is a linear combination of the elements in $\{yh_{w_J}\mid y\in Y_J\}\cup\{h_{xw_J}\mid x\in X_J\backslash Y_J\}$. If $\ell(w)=0,$ then the claim is obvious. Now assume that the claim is true for $z\in X_J$ with $\ell(z)<\ell(w)$.
If $w\in Y_J$, then the claim is clear. If $w\in X_J\backslash Y_J$, using formula (\ref{eq3})  and induction hypothesis we see that the claim is true. Now (c) is proved.

\end{proof}

\begin{Prop}\label{DesEJ}
For $J\subset I$, we have $$E_J=\sum_{w\in Y_J} \mathbb{F} {\bf U}_{w_Jw^{-1}}wC_J,$$
and the set $\{uwC_J \mid w\in Y_J, u\in {\bf U}_{w_Jw^{-1}} \}$ forms a basis of $E_J$.
\end{Prop}

\begin{proof}  For $w\in W$, we set $h'_w=h_w  {\bf 1}_{\op{tr}}\in \mathbb{M}(\op{tr})$. Thus  $h'_{w_J} =(-1)^{\ell(w_J)}\eta_J$ for any $J\subset I$ by \cite[Lemma 2.6 (vi)]{KL}. According to Lemma \ref{wCJ} (c),  we get
$$\mathbb{M}(\op{tr})_J=\sum_{w\in X_J}\mathbb{F}{\bf U}w\eta_J= \sum_{w\in Y_J}\mathbb{F}{\bf U}w\eta_J+ \sum_{x\in X_J \setminus Y_J}\mathbb{F}{\bf U}h'_{xw_J}.$$
We claim that $\mathbb{M}(\op{tr})_J'=\displaystyle \sum_{x\in X_J \setminus Y_J}\mathbb{F}{\bf U}h'_{xw_J} $. For $x\in X_J \setminus Y_J$,   we see that $\mathscr{R}(xw_J)=K$ for some $K \supsetneq J$. Thus $xw_J=yw_K$ for some $y\in X_K$.
By  Lemma \ref{wCJ} (b), we have $ h'_{xw_J}= h'_{yw_K} \in  \mathbb{F} W\eta_K$ which implies  $ \displaystyle \sum_{x\in X_J \setminus Y_J}\mathbb{F}{\bf U}h'_{xw_J}  \subseteq \mathbb{M}(\op{tr})_J'$. On the other hand, we see that $X_K \subseteq X_J\backslash Y_J$ for any  $K \supsetneq J$. Therefore  we get $\mathbb{M}(\op{tr})_K \subseteq  \displaystyle \sum_{x\in X_J \setminus Y_J}\mathbb{F}{\bf U}h'_{xw_J}$ for any  $K \supsetneq J$.
The claim is proved and  we get
$$E_J=\mathbb{M}(\op{tr})_J/\mathbb{M}(\op{tr})_J'= \sum_{w\in Y_J} \mathbb{F} {\bf U} wC_J. $$
It is not difficult to see that $ {\bf U} wC_J= {\bf U}_{w_Jw^{-1}}wC_J$ for any $w\in Y_J$. Thus we obtain the first part.

Now we show that  the set $\{uwC_J \mid w\in Y_J, u\in {\bf U}_{w_Jw^{-1}} \}$ is a basis of $E_J$.  It is enough to prove that this set is linearly independent. Suppose that this set is  linearly dependent. Then there exist  $f_{u, w} \in  \mathbb{F}$ (not all zero) such that
  $$\sum_{w\in Y_J} \sum_{u\in {\bf U}_{w_Jw^{-1}} } f_{u, w} uwC_J=0.$$
Noting that  $E_J=\mathbb{M}(\op{tr})_J/\mathbb{M}(\op{tr})_J'$,  we have
$$ \sum_{w\in Y_J} \sum_{u\in {\bf U}_{w_Jw^{-1}} } f_{u, w} uw\eta_J \in \mathbb{M}(\op{tr})_J'.$$
Without loss of generality, we assume that $u_0=id$ for some $z\in Y_J$ with  $f_{u_0,z} \ne 0$.
Note that the ${\bf T}$-fixed subspace of $\mathbb{M}(\op{tr})_J$ is $\displaystyle \sum_{w\in X_J} \mathbb{F}w\eta_J $.
Since $z\eta_J$ is a ${\bf T}$-stable vector and $\mathbb{M}(\op{tr})_J'=\displaystyle \sum_{x\in X_J \setminus Y_J}\mathbb{F}{\bf U}h'_{xw_J} $,
it is not difficult to see that $z\eta_J$ is a linear combination of the following set
 $$\{w\eta_{J}\mid w\in Y_J, w\ne z\}\cup\{h'_{xw_J}\mid x\in X_J\backslash Y_J\}. $$
This is a contradiction by  Lemma \ref{wCJ} (c). The proposition is proved.
\end{proof}

\begin{Prop}{\cite[Proposition 2.7]{Xi}}\label{EJ}
If $J$ and $K$ are different subsets of $I$, then $E_J$ and $E_K$ are not isomorphic.
\end{Prop}

By the definition of $E_J$, there  exists a filtration of submodules for  $\mathbb{F}[{\bf G}/{\bf B}]$ whose  subquotients are $E_J$ $(J\subset I)$.
In the following of this paper, we prove the irreducibility of  $E_J$ for any $J\subset I$.
Combining Proposition \ref{EJ},  we get Theorem \ref{mainthm}.

\section{Self-enclosed subgroups}

This section contains some preliminaries and properties of unipotent groups which are  useful in later discussion.  As before, let ${\bf U}$ be the  unipotent radical of the Borel subgroup ${\bf B}$. For any $w\in W$, we set $$\Phi_w^-=\{\alpha \in \Phi^+ \mid w(\alpha)\in \Phi^- \}, \ \ \Phi_w^+=\{\alpha \in \Phi^+ \mid w(\alpha)\in \Phi^+ \}.$$
As before,  ${\bf U}_w$ (resp. ${\bf U}_w'$) is  the subgroup of ${\bf U}$ generated by all ${\bf U}_\alpha$  with $\alpha \in\Phi_w^-$ (resp. $\alpha\in\Phi_w^+$). The following properties are well known (see \cite{Car}).

\noindent (a) For $w\in W$ and any root  $\alpha\in \Phi$,  we have $\dot{w}{\bf U}_\alpha \dot{w}^{-1}={\bf U}_{w(\alpha)}$;

\noindent (b)  $ {\bf U}_w$ and ${\bf U}'_w$ are subgroups of  ${\bf U}$, and we have $\dot{w}{\bf U}'_w\dot{w}^{-1} \subset {\bf U}$;

\noindent (c) The multiplication map ${\bf U}_w\times{\bf U}_w'\rightarrow{\bf U}$ is a bijection;

\noindent (d) Let $\Phi^+ =\{\delta_1, \delta_2, \dots, \delta_m\}$. Then
 ${\bf U}= {\bf U}_{\delta_1}{\bf U}_{\delta_2}\dots {\bf U}_{\delta_m}$
and  each element  $u \in {\bf U}$ is uniquely expressible in the form
$u=u_1u_2\dots u_m$ with $u_{i}\in  {\bf U}_{\delta_i}$;

\noindent (e) ({\it Commutator relations}) Given two positive roots $\alpha$ and $\beta$,  there exist a total ordering on $\Phi^+$ and integers $c^{mn}_{\alpha \beta}$ such that
$$[\varepsilon_\alpha(a),\varepsilon_\beta(b)]:=\varepsilon_\alpha(a)\varepsilon_\beta(b)\varepsilon_\alpha(a)^{-1}\varepsilon_\beta(b)^{-1}=
\underset{m,n>0}{\prod} \varepsilon_{m\alpha+n\beta}(c^{mn}_{\alpha \beta}a^mb^n)$$
for all $a,b\in \Bbbk $, where the product is over all
integers $m,n>0$ such that $m\alpha+n\beta \in \Phi^{+}$, taken
according to the chosen ordering.

As before,   let $\Phi^+ =\{\delta_1, \delta_2, \dots, \delta_m\}$ and for an element $u\in {\bf U}$, we have
$u=x_1x_2\dots x_m$ with $x_i\in  {\bf U}_{\delta_i}$. If we choose another order of $\Phi^+$ and write $\Phi^+=\{\delta'_1, \delta'_2, \dots, \delta'_m\}$, we get another expression of $u$ such that
$u=y_1y_2\dots y_m$ with $y_i\in  {\bf U}_{\delta'_i}$. If $\delta_i=\delta'_j =\alpha$ is a simple root, by the commutator relations of root subgroups, we get $x_i=y_j$ which is called the ${\bf U}_{\alpha}$-component of $u$.
Noting that the simple roots are $\Delta=\{\alpha_1, \alpha_2, \dots, \alpha_n\}$ and each $\gamma\in \Phi^+$ can be written as $\gamma=\displaystyle \sum_{i=1}^n k_i \alpha_i$, we denote by $\text{ht}(\gamma)=\displaystyle \sum_{i=1}^n k_i$ the height of $\gamma$. It is easy to see that $\displaystyle \prod_{\text{ht}(\gamma)\geq s} {\bf U}_{\gamma}$ is a subgroup of $\bf U$ for any fixed integer $s\in \mathbb{N}$ by the commutator relations of root subgroups.

Given an order ``$\prec$" on $\Phi^+$,   we list all the positive roots $\delta_1, \delta_2, \dots, \delta_m$ with respect to this order such that
$\delta_i \prec \delta_j$ when  $i<j$. For any $u\in {\bf U}$, we have a unique expression in the form $u=u_1u_2\dots u_m$ with $u_i\in  {\bf U}_{\delta_i}$. Let $X$ be a subset of $\bf U$, we denote by
$$X\cap_{\prec} {\bf U}_{\delta_k}=\{u_{k}\in {\bf U}_{\delta_k} \mid \text{there exists}\  u\in X \ \text{such that}\  u=u_1u_2\dots u_k  \dots u_m\}.$$
It is easy to see that $X\cap {\bf U}_{\delta_k} \subseteq X\cap_{\prec} {\bf U}_{\delta_k}$.
Now let $H$ be a subgroup of ${\bf U}$ and we say that a subgroup $H \subset \bf U$ is self-enclosed with respect to the order ``$\prec$" if
$$H\cap_{\prec} {\bf U}_{\delta_k}= H\cap {\bf U}_{\delta_k} \ \text{for any}\  k=1,2, \dots, m.$$
If $H$ is self-enclosed with respect to any order on $\Phi^+$,  then we say that $H$ is a self-enclosed subgroup of $\bf U$.

Let $H$ be a self-enclosed subgroup of $\bf U$. For each $\gamma \in \Phi^+$, we set
$H_{\gamma}=H \cap {\bf U}_{\gamma}$. Then we have
$$H= H_{\delta_1} H_{\delta_2}\dots  H_{\delta_m}.$$
For $w\in W$, set $H_w= H\cap {\bf U}_w$.  Then it is easy to see that $H_w$ is also a self-enclosed subgroup and we have $\displaystyle H_w= \prod_{\gamma\in \Phi^{-}_w} H_{\gamma}$.

\begin{example}  \normalfont
Suppose $\Bbbk=\bar{\mathbb{F}}_q$ and $\{\delta_1, \delta_2, \dots, \delta_m\}$ are all the positive roots such that
$\text{ht}(\delta_1)\leq \text{ht}(\delta_2) \leq \dots \leq  \text{ht}(\delta_m)$. Assume that ${\bf U}$ is defined over $\mathbb{F}_q$ and  let $U_{q^a}$ be the set of $\mathbb{F}_{q^a}$-points of $\bf U$.
Given $a_1, a_2,\dots, a_m \in \mathbb{N}$ such that $a_i$ is divisible by $a_j$ for any $i < j$,   we  set
$$H= U_{\delta_1, q^{a_1}}  U_{\delta_2, q^{a_2}} \dots  U_{\delta_m, q^{a_m}}.$$
Then it is not difficult to check that $H$ is a self-enclosed subgroup of $\bf U$.
\end{example}

Now let $H$ be a subgroup of $\bf U$. Let ${\bf V}$ be a subgroup of ${\bf U}$ which has the form
${\bf V}= {\bf U}_{\beta_1}{\bf U}_{\beta_2}\dots {\bf U}_{\beta_k}$. We let
$${\bf U}= \displaystyle \bigcup_{x\in L} x  {\bf V} \quad \text{and} \quad {\bf U}= \displaystyle \bigcup_{y\in R}  {\bf V}  y,$$
where $L$ (resp. $R$) is a set of the left (resp. right) coset representatives of ${\bf V}$ in ${\bf U}$. Then we define
the following two sets:
$$H_{{\bf V}}=\{v\in {\bf V} \mid \text{there exists}\  u \in H \ \text{such that} \ u=xv \ \text{for some} \ x\in L\},$$
$${}_{\bf V}H=\{v\in {\bf V} \mid \text{there exists}\  u \in H \ \text{such that} \ u=vy \ \text{for some} \ y\in R\}.$$

\begin{Prop}
Let $H$ be a self-enclosed subgroup of $\bf U$. Let  ${\bf V}$  be a subgroup of  $\bf U$ with the form ${\bf V}= {\bf U}_{\beta_1}{\bf U}_{\beta_2}\dots {\bf U}_{\beta_k}$, where   $\beta_1, \beta_2, \dots, \beta_k\in \Phi^+$. Then we have
$$H_{{\bf V}}= {}_{\bf V}H= H \cap {\bf V}.$$
\end{Prop}

\begin{proof} We just prove that $H_{{\bf V}}= H \cap {\bf V}$. It is clear that $H \cap {\bf V} \subset H_{{\bf V}}$.
Noting that ${\bf V}$ is a subgroup of ${\bf U}$,  we denote $${\bf U}= {\bf U}_{\gamma_1}{\bf U}_{\gamma_2}\dots {\bf U}_{\gamma_l}{\bf U}_{\beta_1}{\bf U}_{\beta_2}\dots {\bf U}_{\beta_k}.$$
Let  $v\in  H_{{\bf V}}$.
Thus there exists  $h \in H$ such that $h=xv$ for some $x\in L$. We write
$$h= x_{\gamma_1}x_{\gamma_2}\dots x_{\gamma_l}v_{\beta_1}v_{\beta_2}\dots v_{\beta_k},  \quad  x_{\gamma_i}\in {\bf U}_{\gamma_i}, v_{\beta_j}\in {\bf U}_{\beta_j}.$$
Since $H$ is self-enclosed, we see that $v_{\beta_j}\in H \cap {\bf U}_{\beta_j}$ which implies that $v\in H \cap {\bf V}$. Therefore we get
$H_{{\bf V}}= H \cap {\bf V}$. Similarly,  we have  ${}_{\bf V}H= H \cap {\bf V}$.  The proposition is proved.
\end{proof}

Now we consider the special case that $\Bbbk$ is a field of positive characteristic $p$. In this case, it is well known that  all the finitely generated subgroups of ${\bf U}$ are finite $p$-groups. We have the following lemma.

\begin{lemma}\label{SCgroup}
Let $X$ be a finite subset of ${\bf U}$. There exists a finite $p$-subgroup $H$ of ${\bf U}$ such that $H\supseteq X$ and $H$ is self-enclosed.
\end{lemma}

\begin{proof} Let  $\Phi^+=  \{\delta_1, \delta_2, \dots, \delta_m\}$ such that
$\text{ht}(\delta_1)\leq \text{ht}(\delta_2) \leq \dots \leq  \text{ht}(\delta_m)$.
For each $1 \leq k \leq m $, we set $X_{k}= X\cap_{\prec} {\bf U}_{\delta_k}$.
Let $H_1$  be the subgroup of ${\bf U}_{\delta_1}$,which is generated by $X_1$. Now we define the subgroup  $H_k$ by recursive step. Suppose that $H_1, H_2,\dots, H_{k-1}$ are defined, we set
$$Y_k= \langle H_1, H_2, \dots, H_{k-1}\rangle  \cap_{\prec} {\bf U}_{\delta_k}$$
and let $H_k$ be the subgroup of ${\bf U}_{\delta_k}$, which is generated by $X_k$ and $Y_k$. Now we have a series of subgroups $ H_1, H_2, \dots, H_{m}$ and then we set
$$H= \langle H_1, H_2, \dots, H_{m}\rangle $$
which is a finitely generated subgroup of ${\bf U}$. Thus $H$ is a finite $p$-subgroup of ${\bf U}$, which contains $X$ by its construction. Moreover, it is not difficult to check that $H$ is a self-enclosed of $\bf U$ using the  commutator relations of root subgroups.

\end{proof}

It is easy to verify that the intersection of  two self-enclosed subgroups of ${\bf U}$ is also  self-enclosed.  For a finite subset  $X$ of ${\bf U}$, there exists a minimal self-enclosed subgroup  $V$ containing $X$.  In this case, we also say that   $V$ is the self-enclosed subgroup generated by $X$.

\section{Non-vanishing property of  the augmentation}

In this section, we fix a subset $J\subset I$. By Proposition \ref{DesEJ}, we have $$E_J=\bigoplus_{w\in Y_J} \mathbb{F} {\bf U}_{w_Jw^{-1}}w C_J$$ as $\mathbb{F}$-vector space. For each $w\in Y_J$, we  denote by

$$\mathfrak{P}_w: E_J \rightarrow \mathbb{F}  {\bf U}_{w_Jw^{-1}} wC_J$$
the projection of vector spaces and by $$\epsilon_w:  \mathbb{F} {\bf U}_{w_Jw^{-1}}\dot{w}C_J \rightarrow \mathbb{F}$$
the augmentation (restricting on $w$) which takes the sum of the coefficients with respect to the natural basis, i.e., for $\xi= \displaystyle \sum_{x\in  {\bf U}_{w_Jw^{-1}} } a_x x  w C_J$, we set $\epsilon_w(\xi)= \displaystyle \sum_{x\in  {\bf U}_{w_Jw^{-1}} } a_x$. Now we denote by
$$\epsilon= \bigoplus_{w\in Y_J} \epsilon_w \mathfrak{P}_w:  E_J \rightarrow \mathbb{F}^{|Y_J|}$$
the augmentation on $E_J$.

When considering the irreducibility of Steinberg module, the non-vanishing property of the augmentation  is very crucial (see \cite[Lemma 2.5]{Yang} and \cite[Proposition 1.6]{PS}).  In this section, we show that the non-vanishing property also holds for the  augmentation $\epsilon$ defined above. Firstly we have the following lemma.

\begin{lemma}\label{firstlemma}
Let $\xi \in  E_J $ be a nonzero element. Then there exists  $g\in {\bf G}$ such that $\mathfrak{P}_e(g\xi)$ is nonzero.
\end{lemma}

\begin{proof} By Proposition \ref{DesEJ},   $\xi \in  E_J $  has  the following  expression
$$\xi= \sum_{w\in Y_J} \sum_{x\in {\bf U}_{w_Jw^{-1}}}a_{w,x} x w C_J.$$
Then there exists an element $h\in W$ with minimal length such that $a_{h,x}\ne 0$ for some $x\in  {\bf U}_{w_Jh^{-1}}$, which implies that $\mathfrak{P}_h(\xi)$ is nonzero. When $h=e$, the lemma is proved. Now suppose that $\ell(h)\geq 1$, so there is a simple reflection $s$ such that $\sigma=sh< h$. Without loss of generality, we can assume that $a_{h,id}\ne 0$.
We claim that either $\mathfrak{P}_{\sigma}(\dot{s} \xi) $ is nonzero or $\mathfrak{P}_{\sigma}(\dot{s}y \xi)$ is nonzero for some $y \in {\bf U}_s$.

If $\mathfrak{P}_{\sigma}(\dot{s} \xi) =0$, then according to Lemma \ref{suhlemma},  there exists at least one element  $v\in Y_J$ which satisfies  the following condition
$$(\spadesuit) \ sv \notin Y_J \  \text{and} \ \mathfrak{P}_{\sigma}(svC_J)\ne 0.$$
The subset of  $Y_J$ whose elements satisfy this condition is also denoted by $\spadesuit$. Thus $\mathfrak{P}_{\sigma}(\dot{s} \xi) =0$ tells us that
$$\mathfrak{P}_{\sigma}(\dot{s} \cdot \mathfrak{P}_h(\xi))+  \mathfrak{P}_{\sigma}(\dot{s} \cdot \sum_{v\in \spadesuit}\mathfrak{P}_v(\xi))=0.$$
In particular, we get $\displaystyle \mathfrak{P}_{\sigma}(\dot{s} \cdot \sum_{v\in \spadesuit}\mathfrak{P}_v(\xi)) \ne 0$.
Since ${\bf U}$ is infinite, there exists infinitely many $y\in {\bf U}_s$ such that the $ {\bf U}_s$-component of $yx$ is nontrivial for any $x$ with $a_{h, x} \ne 0$. For such an element $y$, we get  $\mathfrak{P}_{\sigma}(\dot{s} \cdot \mathfrak{P}_h(y\xi))=0$ by Lemma \ref{suhlemma} (b).

On the other hand, for $v\in \spadesuit$  and $a_{v,x}\ne 0$, we see that  the $ {\bf U}_s$-component of $x$ is trivial, i.e., $ x\in {\bf U}'_s$.
Note that ${\bf U}_{w_J \sigma^{-1} s}={( {\bf U}_{w_J \sigma^{-1}})}^{s} \cdot  {\bf U}_s$ and $ {\bf U}'_{w_J \sigma^{-1}} =  {({\bf U}'_{w_J \sigma^{-1}s})}^{s} \cdot {\bf U}_s$.  Then we can write $$x=n(x)p(x),   \quad  \text{where} \ n(x)\in   {({\bf U}_{w_J\sigma^{-1}})}^{s}  \ \text{and} \ p(x) \in  {\bf U}'_{w_J\sigma^{-1}s}.$$
Since this expression is unique,  we can regard $p(-)$ and $n(-)$ as functions on ${\bf U}'_s$. We let $yx= \omega_y(x)y$, where $\omega_y(x) \in {\bf U}'_s$. Using the commutator relations of root subgroups, we can choose $y$ such that
$$n(\omega_y(x'))\ne n(\omega_y(x))$$
unless $n(x)=n(x')$ since there are only finitely many $x's$ satisfying $a_{v,x} \ne 0$.
Therefore if we write $$ \mathfrak{P}_{\sigma}(\dot{s} \cdot \sum_{v\in \spadesuit}\mathfrak{P}_v(\xi))=\sum  b_{\sigma,x}   n(x)^{\dot{s}} \sigma C_J\ne 0,$$
it is not difficult to see that
$$\mathfrak{P}_{\sigma}(\dot{s} \cdot \sum_{v\in \spadesuit}\mathfrak{P}_v(y\xi))= \sum  b_{\sigma,x} n({\omega_y(x)})^{\dot{s}}\sigma C_J$$
which is also nonzero.  Therefore  $$\mathfrak{P}_{\sigma}(\dot{s}y \xi)= \mathfrak{P}_{\sigma}(\dot{s} \cdot \sum_{v\in \spadesuit}\mathfrak{P}_v(y\xi))\ne 0.$$
By the argument above, we can do induction on the length of $h$ and thus the lemma is proved.
\end{proof}

The  non-vanishing property of the augmentation $\epsilon$ on  $E_J$ is as follows:

\begin{Prop}\label{Key}
Let $\xi \in  E_J $ be a nonzero element. Then there exists $g\in {\bf G}$ such that  $\epsilon(g\xi)$ is nonzero.
\end{Prop}

\begin{proof} By Lemma \ref{firstlemma}, we can assume that $\mathfrak{P}_e(\xi)$ is nonzero. For $$\xi= \sum_{w\in Y_J} \sum_{x\in {\bf U}_{w_Jw^{-1}}}a_{w,x} x w C_J \in  E_J,$$
we say that $\xi$ satisfies the condition $\heartsuit_h$  if $\displaystyle \sum_{x\in {\bf U}'_h} a_{e,x}\ne 0$ for some $h\in W_J$. We prove the following claim: if $\xi$ satisfies the condition $\heartsuit_h$ for some $h\in W_J$, then there exists $g\in {\bf G}$ such that  $\epsilon_e \mathfrak{P}_e(g\xi)$ is nonzero.

We prove this claim  by induction on the length of $h$. If $h=e$, then it is obvious that $\epsilon_e \mathfrak{P}_e(\xi)= \displaystyle \sum_{x\in {\bf U}_{w_J}} a_{e,x}$  which  is already nonzero. We assume that the claim is valid for any $h\in W_J$ with $\ell(h)\leq m$. Now let $h\in W_J$ with $\ell(h)=m+1$ such that  $\displaystyle \sum_{x\in {\bf U}'_h} a_{e,x}\ne 0$.  We have $h=\tau s$ for some $s\in \mathscr{R}(h)$.
Then ${\bf U}_h= {\bf U}^s_\tau \cdot  {\bf U}_s$ and  ${\bf U}'_{\tau}=({\bf U'_h})^s \cdot {\bf U}_s$ by definition.
Now our aim is to show that there exists $g\in {\bf G}$ such that $g\xi$  satisfies the condition $\heartsuit_\tau$.

Firstly we prove  that the element $\dot{s} \cdot \mathfrak{P}_e(\xi)$
satisfies the condition $\heartsuit_\tau$. Since ${\bf U}_{w_J}= {\bf U}'_h{\bf U}_h= {\bf U}'_h{\bf U}^s_\tau {\bf U}_s $, each element $x\in {\bf U}_{w_J}$ has a unique expression
$$x=x'_h x_{\tau} x_s, \quad x'_h \in{\bf U}'_h, x_{\tau}\in {\bf U}^s_\tau , x_s \in  {\bf U}_s. $$
We just need to consider the coefficients of $a_{e, x}$ with $\dot{s} x'_h x_{\tau}  \dot{s}^{-1} \in {\bf U}'_{\tau}$, which implies that $x_{\tau}=id$.
For the case $x_s\ne id$, using Lemma \ref{suhlemma} (c), we have
$$\dot{s} x C_J=x''_h \dot{s}x_s  C_J =  x''_h (f(x_s)-1) C_J, $$
 where  $x''_h= \dot{s} x'_h \dot{s}^{-1} \in ({\bf U'_h})^s \subset {\bf U}'_{\tau}$ and  $f(x_s)\in {\bf U}_s$. Therefore if we write
 $$\dot{s} \cdot \mathfrak{P}_e(\xi)= \sum_{x\in {\bf U}_{w_J}}b_{e,x} x C_J, $$
 then $\displaystyle \sum_{x\in {\bf U}'_\tau} b_{e,x} =  -\sum_{x\in {\bf U}'_h} a_{e,x}\ne 0$. Thus  $\dot{s} \cdot \mathfrak{P}_e(\xi)$
 satisfies the condition $\heartsuit_\tau$.

Now we consider
$\dot{s}\xi$ and if $\dot{s}\xi$ satisfies the condition $\heartsuit_\tau$,  we are done. Otherwise there exists at least one element  $v\in Y_J$ which  satisfies  the following condition
$$(\clubsuit): \ sv \notin Y_J \  \text{and} \ \mathfrak{P}_{e}(svC_J)\ne 0.$$
The subset of  $Y_J$ whose elements satisfy this condition is also denoted by $\clubsuit$.
With this setting,
$\displaystyle \dot{s} \cdot \mathfrak{P}_e(\xi)+  \mathfrak{P}_{e}(\dot{s} \cdot \sum_{v\in \clubsuit}\mathfrak{P}_v(\xi))$ does not satisfy the condition $\heartsuit_\tau$, which implies that $\displaystyle  \mathfrak{P}_{e}(\dot{s} \cdot \sum_{v\in \clubsuit}\mathfrak{P}_v(\xi))$ satisfies the condition $\heartsuit_\tau$ since we have proved that $\dot{s} \cdot \mathfrak{P}_e(\xi)$ satisfies the condition $\heartsuit_\tau$.
 Since ${\bf U}$ is infinite, we can choose an element $y\in {\bf U}_s$ such that  the  $ {\bf U}_s$-component of $yx$  is nontrivial for any $x$ with $a_{e, x} \ne 0$. Then we consider the element $\dot{s}y\xi$. Using Lemma \ref{suhlemma} (c), it is easy to see that
$\dot{s}\cdot \mathfrak{P}_e(y\xi)$ does not satisfy the condition $\heartsuit_\tau$.

Now for  $v\in \clubsuit$  and $x\in {\bf U}_{w_Jv^{-1}}$ with $a_{v,x}\ne 0$,   noting that  the ${\bf U}_s$-component of $x$ is trivial,  we write
$$x= m(x) q(x),  \quad \text{where} \ m(x) \in {\bf U}_{w_Js},  q(x)\in ({\bf U}'_{w_J})^s.$$
For $y\in  {\bf U}_s$, using the  commutator relations of root subgroups,  we have
$$ym(x)= m_{y}(x) y,  \quad  \text{where}\  m_y(x) \in {\bf U}_{w_Js},$$
and
$$ yq(x)=q_y(x) y,  \quad \text{where}\    q_y(x)  \in  ({\bf U}'_{w_J})^s.$$
Since ${\bf U}'_{\tau}= ({\bf U}'_h)^s {\bf U}_s$,  we get $m(x)^{\dot{s}}\in {\bf U}'_{\tau}$  if  and only if  $m(x)\in ({\bf U}'_h)^s $.
Thus  $m_y(x)^{\dot{s}} \in {\bf U}'_{\tau}$ if and only if $m(x)^{\dot{s}} \in {\bf U}'_{\tau}$.
Therefore if we write $$ \mathfrak{P}_{e}(\dot{s} \cdot \sum_{v\in\clubsuit}\mathfrak{P}_v(\xi))=\sum  b_{x}   m(x)^{\dot{s}} C_J,$$
it is not difficult to see that
$$ \mathfrak{P}_{e}(\dot{s} \cdot \sum_{v\in \clubsuit}\mathfrak{P}_v(y\xi))=\sum  b_{x}   m_y(x)^{\dot{s}} C_J.$$
Noting that $\displaystyle  \mathfrak{P}_{e}(\dot{s} \cdot \sum_{v\in \clubsuit}\mathfrak{P}_v(\xi))$ satisfies the condition $\heartsuit_\tau$, we see that
$\displaystyle  \mathfrak{P}_{e}(\dot{s} \cdot \sum_{v\in \clubsuit}\mathfrak{P}_v(y\xi))$ satisfies the condition $\heartsuit_\tau$.
Finally there exists $g\in {\bf G}$ such that $g\xi$  satisfies the condition $\heartsuit_e$, which implies that  $\epsilon_e \mathfrak{P}_e(g\xi)$ is nonzero.  We have proved our claim.

Now we can assume that $a_{e,id}\ne 0$. Thus the element $\xi$ satisfies the condition $\heartsuit_{w_J}$. According to our claim,  there exists $g\in {\bf G}$ such that  $\epsilon_e \mathfrak{P}_e(g\xi)$ is nonzero. In particular, $\epsilon(g\xi)$ is nonzero and the proposition is proved.
\end{proof}

\section{Proof of the main theorem}
In this section, we give the proof of Theorem \ref{mainthm}. Firstly,  we deal with the cases:   (1) $\text{char} \ \Bbbk=0$;  (2) $\text{char} \ \Bbbk>0$
and $\text{char} \ \Bbbk\ne \text{char} \ \mathbb{F}$.  For $J\subset I$,  we show that any nonzero submodule $M$ of $E_J$ contains $C_J$, and hence $M=E_J$.
In particular,  $E_J$ is  irreducible for any $J\subset I$.
Let $\xi \in M$ be a nonzero element with the following expression
$$\xi= \sum_{w\in Y_J} \sum_{x\in {\bf U}_{w_Jw^{-1}}}a_{w,x} x w C_J \in M.$$
By Proposition \ref{Key}, we can assume that $\epsilon (\xi) \ne 0$.
In the case (1) by \cite[Proposition 5.4]{PS} and in the case (2)  by   \cite[Proposition 6.7]{PS},   we have
$$\sum_{w\in Y_J}\sum_{x\in {\bf U}_{w_Jw^{-1}}}a_{w,x}  w C_J \in M.$$
In particular, we see that
$$M\cap \sum_{w\in Y_J}\mathbb{F} w C_J\neq0.$$
Noting that the discussion in the proof of  \cite[Claim 2]{CD1} is still valid in our general setting,  we see that  $E_J$ is irreducible for any $J\subset I$.
The $E_J$'s are  pairwise non-isomorphic by Proposition \ref{EJ}.

\bigskip

It remains to consider the case $\text{char} \ \Bbbk=  \text{char} \ \mathbb{F}=p>0$. From now on, we assume that $\text{char} \ \mathbb{F}=  \text{char} \ \Bbbk=p$. For any finite subset $X$ of ${\bf G}$, let $\underline{X}:=\sum_{x\in X}x \in \mathbb{F} {\bf G}$. The following lemma is easy to get and will   be very useful in our later discussion.

\begin{lemma} \label{easylemma}
Let $P$ be a finite abelian $p$-group  such that  $P=H\times K$, where $H, K$ are two subgroups of $P$. Let $H'$ be a subgroup of $P$ such that $|H'|=|H|$. Then $\underline{H'}\ \underline{K}= 0$ or  $\underline{P}$.

\end{lemma}

 For a  self-enclosed subgroup $H$  of ${\bf U}$,  set  $H_{\gamma}=H \cap {\bf U}_{\gamma}$ as before for each $\gamma \in \Phi^+$. Let $\Phi^+=  \{\delta_1, \delta_2, \dots, \delta_m\}$.  We have
 $$\underline{H}= \underline{H_{\delta_1}}\  \underline{H_{\delta_2}}\ \dots \ \underline{H_{\delta_m}}.$$
Let $H_w= H\cap {\bf U}_w$.  Then we have
$\displaystyle H_w= \prod_{\gamma\in \Phi^{-}_w} H_{\gamma}$ and $\displaystyle \underline{H_w}= \prod_{\gamma\in \Phi^{-}_w} \underline{H_{\gamma}}.$
The following two lemmas are very crucial in the later proof of Theorem \ref{mainthm}.

\begin{lemma} \label{oneterm}
 Assume that $\text{char} \ \mathbb{F}=  \text{char} \ \Bbbk=p>0$ and  let $M$ be a nonzero $\mathbb{F}{\bf G}$-submodule of $E_J$. Then there exist an element $w\in Y_J$  and a  finite $p$-subgroup $X$ of ${\bf U}_{w_Jw^{-1}}$ such that
$\underline{X} wC_J \in M$.

\end{lemma}

\begin{proof} Let $\xi$ be a nonzero element of $M$ which has the form
$$\xi= \sum_{w\in Y_J} \sum_{x\in {\bf U}_{w_Jw^{-1}}}a_{w,x} x w C_J \in  E_J.$$
 By Lemma  \ref{SCgroup}, there exists a self-enclosed finite $p$-subgroup $V$ of  ${\bf U}$, which contains all $x\in {\bf U}_{w_Jw^{-1}}$ with  $a_{w,x}\ne 0$. Then we have
$$ \mathbb{F} V \xi\subset  \bigoplus_{w\in Y_J} \mathbb{F}  V_{w_Jw^{-1}} w C_J $$
as $\mathbb{F}  V$-modules. Since $(\mathbb{F} V \xi)^{V}\ne 0$ by \cite[Proposition 26]{Se} and noting that
$$(\bigoplus_{w\in Y_J} \mathbb{F} V_{w_Jw^{-1}} w C_J)^{V}\subset  \bigoplus_{w\in Y_J}\mathbb{F}
\underline{V_{w_Jw^{-1}}} w C_J,$$
there exists a nonzero element
$$\eta= \sum_{w\in Y_J} a_w \underline{V_{w_Jw^{-1}}} w C_J\in  \mathbb{F}  V \xi \subset  M.$$
Set $A(\eta)=\{w\in Y_J\mid a_w\ne 0\}$. If $|A(\eta)|=1$, the lemma is proved.

Now we assume that $|A(\eta)|\geq 2$.  We set  $\Phi(\eta)=\displaystyle \bigcup_{w\in A(\eta) } \Phi_{w_Jw^{-1}}^- $.
Let $\Phi(\eta)=\{\gamma_1, \gamma_2, \dots, \gamma_d\}$ such that $\text{ht}(\gamma_1)\leq \text{ht}(\gamma_2) \leq \dots \leq \text{ht}(\gamma_d)$.  Let $s$ be the maximal integer such that
$\gamma_s\notin \displaystyle \bigcap_{w\in A(\eta) } \Phi_{w_Jw^{-1}}^- $.
Let  $y\in {\bf U}_{\gamma_s}\backslash V_{\gamma_s}$ and  $H$ be a self-enclosed finite $p$-subgroup of ${\bf U}_{\gamma_s}{\bf U}_{\gamma_{s+1}}\dots {\bf U}_{\gamma_d}$ such that $H$ contains $V_{\gamma_s}V_{\gamma_{s+1}}\dots V_{\gamma_d}$ and $y$.
Let $X$ be the  self-enclosed  subgroup of $\bf U$ which is generated by $H$ and $V$.
Then it is easy to see that  $X$ has the following form
$$X=V_{\gamma_1} \dots V_{\gamma_{s-1}} X_{\gamma_s} \dots X_{\gamma_d},$$
where $X_{\gamma_k}= X\cap {\bf U}_{\gamma_k}$ for $s\leq k\leq d$.
Denote by  $\Omega_s$  a set of the left coset representatives of $V_{\gamma_s}V_{\gamma_{s+1}}\dots V_{\gamma_d}$  in $X_{\gamma_s} X_{\gamma_{s+1}} \dots X_{\gamma_d}$.
For the $w\in Y_J$ such that $\gamma_s \in  \Phi_{w_Jw^{-1}}^- $, we have
$$\underline{\Omega_s} \ \underline{V_{w_Jw^{-1}}}  w C_J=  \underline{X_{w_Jw^{-1}}} w C_J.$$
For the $w\in Y_J$ such that $\gamma_s \notin  \Phi_{w_Jw^{-1}}^- $, we have
$$\underline{\Omega_s} \ \underline{V_{w_Jw^{-1}}}  w C_J= 0$$
since $\text{char}\ \mathbb{F} =p$. Then we get
$$\eta'= \underline{\Omega_s} \ \eta= \sum_{w\in Y_J} b_w \underline{X_{w_Jw^{-1}}} w C_J,$$
which satisfies that $|A(\eta')| < |A(\eta)|$, where $A(\eta')=\{w\in Y_J\mid b_w\ne 0\} $. Thus by the induction on the cardinality of  $A(\eta)$, the lemma is proved.

\end{proof}

\begin{lemma} \label{leastterm}
 Assume that $\text{char} \ \mathbb{F}=  \text{char} \ \Bbbk=p>0$ and   let $M$ be a nonzero $\mathbb{F} {\bf G}$-submodule of $E_J$. If there exists a finite $p$-subgroup $X$ of ${\bf U}_{w_Jw^{-1}s}$ such that $\underline{X} swC_J \in M$, where $sw\in Y_J$ and $sw>w$ (which implies that $w\in Y_J$), then there exists a finite $p$-subgroup $H$ of  ${\bf U}_{w_Jw^{-1}}$ such that $\underline{H} wC_J \in M$.

 \end{lemma}

\begin{proof} Using Lemma \ref{SCgroup}, we can assume that $X$ is a self-enclosed subgroup of ${\bf U}_{w_Jw^{-1}s}$. Since ${\bf U}_{w_Jw^{-1}s}= {\bf U}_{s}  ({\bf U}_{w_Jw^{-1}})^s $, we can write
$X= X_{\alpha} V $, where $V= X\cap ({\bf U}_{w_Jw^{-1}})^s$ is also a self-enclosed subgroup of $({\bf U}_{w_Jw^{-1}})^s$. Thus we have $\underline{X}= \underline{X_{\alpha}}\ \underline{V}$.
In the following, we will prove that if $\underline{Y}\ \underline{V} \ swC_J \in M$ for some finite subset $Y$ of $ {\bf U}_{s}$ and a  self-enclosed subgroup $V$ of $({\bf U}_{w_Jw^{-1}})^s$, then  there exists a finite $p$-subgroup $H$ of  ${\bf U}_{w_Jw^{-1}}$ such that $\underline{H} wC_J \in M$. Without loss of generality, we can assume that  $Y$ contains the neutral element of ${\bf U}_s$.

For each $u\in {\bf U}_{\alpha}\backslash\{id\}$, we have
$$\dot{s}u \dot{s}= f_\alpha(u)h_{\alpha}(u) \dot{s} g_\alpha(u),$$
where $f_\alpha(u), g_\alpha(u) \in {\bf U}_{\alpha}$ and $ h_{\alpha}(u)\in {\bf T}$ are uniquely determined.
Then
$$\dot{s}u \underline{V} swC_J= f_\alpha(u) h_{\alpha}(u)\dot{s}  g_\alpha(u) \dot{s}^{-1}\underline{V} swC_J.$$
Without loss of generality, we can assume that the group $V$ contains enough elements such that
$$g_\alpha(u) \dot{s}^{-1}\underline{V} swC_J=  \dot{s}^{-1}\underline{V} swC_J$$
for any $u\in Y\backslash\{id\}$.  Indeed, we let
$$G_{\alpha}(X)=\{g_\alpha(u) \in {\bf U}_{\alpha}\mid u\in Y \backslash\{id\}\}$$
and  $H$ be a self-enclosed subgroup which contains $G_{\alpha}(X)$ and $\dot{s}^{-1}V  \dot{s}$.
Then $H_{w_Jw^{-1}}= H \cap {\bf U}_{w_Jw^{-1}}$ is also a self-enclosed subgroup which contains $\dot{s}^{-1}V  \dot{s}$.
Then we can consider $ \underline{Y}\ \underline{\dot{s}H_{w_Jw^{-1}}\dot{s}^{-1}}$ instead of $\underline{Y}\ \underline{V}$ from the beginning.
Noting that $h_{\alpha}(u) \in {\bf T}$,  we have
$$\dot{s}u \underline{V} swC_J= f_\alpha(u) h_{\alpha}(u)\underline{V} swC_J= f_\alpha(u)\underline{ {h_{\alpha}(u)Vh_{\alpha}(u)}^{-1}} swC_J,$$
which implies that
$$\dot{s}\underline{Y}\ \underline{V}  swC_J = \underline{\dot{s}V \dot{s}^{-1}} wC_J+ \sum_{u\in  Y\backslash\{id\} }f_\alpha(u)\underline{ {h_{\alpha}(u)Vh_{\alpha}(u)}^{-1}} swC_J.$$

Now we let
$$\Phi^-_{w_Jw^{-1}} \cup \Phi^-_{w_Jw^{-1}s}=\{\beta_1=\alpha, \beta_2,\dots, \beta_m\}$$
such that  $\text{ht}(\beta_1) \leq \text{ht}(\beta_2)\leq \dots \leq \text{ht}(\beta_m)$. Since $sw\in Y_J$ and $sw>w$, we have $ ({\bf U}_{w_Jw^{-1}})^s \ne  {\bf U}_{w_Jw^{-1}}$ by \cite[Corollary 2.2]{CD2}. Now let $r$ be the maximal integer such that  $\beta_r\notin \Phi^-_{w_Jw^{-1}} \cap \Phi^-_{w_Jw^{-1}s} $ and $\beta_j\in \Phi^-_{w_Jw^{-1}} \cap \Phi^-_{w_Jw^{-1}s} $ for $j>r$.
When $\beta_r \in \Phi^-_{w_Jw^{-1}}  \backslash \Phi^-_{w_Jw^{-1}s}$, using Lemma \ref{SCgroup}, Lemma \ref{easylemma} and \cite[Lemma 4.5]{CD2}, we can choose certain subgroup  $\Omega_k$ of ${\bf U}_{\beta_k}$ for each $r\leq k \leq m$ such that
$$\underline{\Omega_r}\ \underline{\Omega_{r+1}} \dots \underline{\Omega_m} \ f_\alpha(u)\underline{ {h_{\alpha}(u)Vh_{\alpha}(u)}^{-1}} swC_J=0$$
for any $u\in  Y\backslash\{id\}$  and $$\underline{\Omega_r}\ \underline{\Omega_{r+1}} \dots \underline{\Omega_m} \ \underline{\dot{s}V \dot{s}^{-1}} wC_J= \underline{\Omega}  wC_J$$ for some finite subgroup $\Omega$ of ${\bf U}_{w_Jw^{-1}}$.  Then the lemma is proved in this case.

When $\beta_r \in \Phi^-_{w_Jw^{-1}s} \backslash \Phi^-_{w_Jw^{-1}}$, also by Lemma \ref{SCgroup},  Lemma \ref{easylemma} and \cite[Lemma 4.5]{CD2}, we can choose certain subgroup  $\Gamma_k$ of ${\bf U}_{\beta_k}$ for each $r\leq k \leq m$ such that there exists at least one $u\in Y \backslash\{id\}$ which  satisfies
$$\underline{\Gamma_r}\ \underline{\Gamma_{r+1}} \dots \underline{\Gamma_m} \ f_\alpha(u)\underline{ {h_{\alpha}(u)Vh_{\alpha}(u)}^{-1}} swC_J= f_\alpha(u) \underline{\Gamma}swC_J, $$
where $\Gamma$ is some finite subgroup of $({\bf U}_{w_Jw^{-1}})^s $. On the other hand, these groups $\Gamma_r, \Gamma_{r+1}, \dots, \Gamma_m$   also make
$$\underline{\Gamma_r}\ \underline{\Gamma_{r+1}} \dots \underline{\Gamma_m}\ \underline{\dot{s}V \dot{s}^{-1}} wC_J= 0.$$
Therefore we get  $\displaystyle \sum_{x\in F}x \underline{\Gamma} swC_J\in M$ for some set $F$ with $|F|< |Y|$ and some finite subgroup $\Gamma$ of $({\bf U}_{w_Jw^{-1}})^s $. Hence by the same discussion as before, we  get another element $\displaystyle \sum_{y\in F'} y \underline{\Gamma'} swC_J\in M$ for some set $F'$ with $|F'|< |F|$ and some finite subgroup $\Gamma'$ of $({\bf U}_{w_Jw^{-1}})^s $. Finally, we get an element $\underline{K} swC_J\in M$
 for some finite subgroup $K$ of $({\bf U}_{w_Jw^{-1}})^s $. Thus we have $\underline{K^{\dot{s}}}  wC_J\in M$ and the lemma is proved.

\end{proof}

Finally we prove the irreducibility of $E_J$ in  the case $\text{char} \ \mathbb{F}=  \text{char} \ \Bbbk=p>0$ using the previous lemmas. Let $M$ be a nonzero $\mathbb{F} {\bf G}$-submodule of $E_J$. Combining Lemma \ref{oneterm} and Lemma \ref{leastterm},  there exists a  finite $p$-subgroup $H$ of ${\bf U}_{w_J}$ such that $\underline{H}C_J\in M$. Similar to the arguments of \cite[Lemma 2.5]{Yang}, we see that the sum of all coefficients of  $\dot{w_J}x C_J$ in terms the basis
$\{uC_J\mid u\in {\bf U}_{w_J}\}$ is zero when $x$ is not the  neutral element of ${\bf U}_{w_J}$. So if we write
$$\xi= w_J \underline{H}C_J= \displaystyle  \sum_{x\in {\bf U}_{w_J} } a_x xC_J,$$
we have $\displaystyle \sum_{x\in {\bf U}_{w_J} } a_x= (-1)^{\ell(w_J)}$ which is nonzero. We consider the $\mathbb{F} {\bf U}_{w_J}$-module generated by $\xi$, and then using \cite[Proposition 4.1]{PS}, we see  that $C_J\in M$. Therefore $M=E_J$, which implies the  irreducibility  of $E_J$ for any $J\subset I$. All the $\mathbb{F}{\bf G}$-modules  $E_J$  are  pairwise non-isomorphic by Proposition \ref{EJ} and thus Theorem \ref{mainthm} is proved.

\bigskip

\noindent{\bf Acknowledgements.} The author is grateful to  Nanhua Xi and  Xiaoyu Chen for their suggestions and helpful discussions.
The author also thanks the referee for  the helpful comments which  greatly improve the manuscript.  The work is sponsored by Shanghai Sailing Program (No.21YF1429000) and NSFC-12101405.

\bigskip

\noindent{\bf Statements and Declarations }\ \ The author declares that he has no conflict of interests with others.

\bigskip

\bibliographystyle{amsplain}

\end{document}